\documentclass[12pt]{amsart}
\usepackage{amsmath,amsthm,amssymb,mathtools,graphicx,bm,caption,enumitem,epstopdf,mathscinet}
\numberwithin{equation}{section}
\newtheorem{thm}{Theorem}[section]
\newtheorem{lem}[thm]{Lemma}
\newtheorem{cor}[thm]{Corollary}
\newtheorem{prop}[thm]{Proposition}
\newtheorem*{thm*}{Theorem}
\newtheorem*{claim}{Claim}
\theoremstyle{definition}
\newtheorem{defn}[thm]{Definition}
\theoremstyle{remark}
\newtheorem{remark}[thm]{Remark}
\allowdisplaybreaks
\renewcommand{\restriction}{\mathord{\upharpoonright}}
\DeclareMathOperator{\Leb}{Leb}
\DeclareMathOperator{\Lyap}{Lyap}

\begin{document}

\title{Schmidt's Game and Nonuniformly Expanding Interval Maps}
\author{Jason Duvall}
\address{Dept.\ of Mathematics, University of Houston, Houston, TX 77204}
\email{jrduvall@math.uh.edu}

\begin{abstract}
	We study a class of two-branched nonuniformly expanding interval maps with a neutral fixed point at 0, a class that includes Manneville--Pomeau maps.  We prove that the set of points whose forward orbits miss an interval with left endpoint 0 is strong winning for Schmidt's game.  Strong winning sets are dense, have full Hausdorff dimension, and satisfy a countable intersection property.  Similar results were known for certain expanding maps, but these did not address the nonuniformly expanding case.  Our analysis is complicated by the presence of infinite distortion and unbounded geometry.
\end{abstract}

\maketitle

\section{Introduction and statement of results} \label{sec:intro}
Let $X$ be a compact metric space, $f$ a countably-branched piecewise-continuous map, and $\mu$ an $f$-invariant measure on $X$.  There are broad conditions under which $\mu$-almost every point in $X$ has dense forward orbit under $f$.  This is the case, for example, if $\mu$ is ergodic and fully supported on $X$.  The ``exceptional sets'' of points with nondense orbits, despite being $\mu$-null, are nevertheless often large in a different sense.  In particular they are often winning for Schmidt's game, which implies that they are dense in $X$, have full Hausdorff dimension (if $X \subset \mathbb{R}^n$), and remain winning when intersected with countably many suitable winning sets in $X$ (see Theorem \ref{thm:schmidt} for a precise result and \S\ref{sec:Schmidt} for definitions).  Examples of systems possessing winning exceptional sets include surjective endomorphisms of the torus \cite{MR2818688,MR980795}, beta transformations \cite{MR2660561,MR3206688}, the Gauss map \cite{MR195595}, and $C^2$ (uniformly) expanding self-maps of the circle \cite{MR2480100}.

In this article we add to this list a certain collection of nonuniformly expanding interval maps possessing a neutral fixed point at 0.  Throughout this paper we assume that $f \colon [0,1] \to [0,1]$ is a two-branched interval map and that there exists $r_1 \in (0,1)$ such that the following conditions hold (using one-sided derivatives as needed):
\begin{enumerate}
	\item \label{condfirst} $f((0,r_1)) = f((r_1,1)) = (0,1)$;
	\item \label{cond2} $f$ is $C^1$ on $[0,r_1)$ and $C^2$ on $(0,r_1) \cup (r_1,1]$;
	\item \label{cond3} $f(0) = f(r_1) = 0$, $f'(0) = 1$, and $\lvert f' \rvert > 1$ on $(0,1]$ 
	(hence $\inf_{x \in [\epsilon,1]} \, \lvert f' \rvert > 1$ for all $\epsilon > 0$);
	\item \label{cond4} $\sup_{x \in [0,1]} \, \lvert f'(x) \rvert < \infty$ and 
	$\sup_{x \in (r_1,1]} \, \lvert f''(x) \rvert < \infty$;
	\item \label{condlast} there exist $C \geq 1$ and $\gamma > 0$ such that for $x \in (0,r_1)$,
	\begin{equation*}
		C^{-1} \leq \frac{f''(x)}{x^{\gamma-1}} \leq C.
	\end{equation*}
\end{enumerate}

If one wishes to relax assumption (\ref{cond3}) by allowing $\lvert f'_+ (r_1) \rvert = 1$, all results in this article still hold as long as the induced first return map to $[r_1,1]$ is uniformly expanding (see the next section for the definition of the first return map).
\begin{remark}
	Conditions (\ref{cond2}) and (\ref{cond3}) (even a weakened condition (\ref{cond3})) imply that $f$ is monotone on $(r_1,1]$.  Henceforth, until the final section, we assume that $f$ is increasing on $(r_1,1]$ in order to simplify the presentation.  In \S\ref{sec:decreasing} we detail the modifications to our proofs necessary in order to accommodate the case when $f$ is decreasing on $(r_1,1]$.
\end{remark}

The standard example from this class is the Manneville--Pomeau map
\begin{equation*}
f(x) = \begin{cases}
x+x^{1+\gamma} & \text{if } \, 0 \leq x < r_1 \\
x+x^{1+\gamma} - 1 & \text{if } \, r_1 \leq x \leq 1,
\end{cases}
\end{equation*}
where $r_1$ is the unique solution of $x+x^{1+\gamma} = 1$ (see Figure \ref{fig:MP}).  Our main result is the following theorem, which we prove in \S\ref{sec:corollary}.

\begin{thm} \label{thm:maintheoremf}
	For $f$ satisfying conditions (\ref{condfirst})--(\ref{condlast}) above, the set
	\begin{equation*}
	\mathcal{E}_f := \{ x \in [ 0,1 ] \colon  [ 0,\epsilon  ) \cap \{ f^n x \}_{n \geq 0} = \emptyset \textrm{ for some } \epsilon > 0 \}
	\end{equation*}
	is strong winning for Schmidt's game.
\end{thm}

\begin{remark}
	As the proof of Theorem \ref{thm:maintheoremf} will demonstrate, the strong-winning dimension of $\mathcal{E}_f$, i.e., the supremum of all $\alpha$ for which $\mathcal{E}_f$ is $\alpha$-strong winning, depends on $\gamma$.  
\end{remark}

\begin{remark}
	It is well-known that $\Leb ( \mathcal{E}_f ) = 0$.  Indeed, we may express $\mathcal{E}_f$  as a countable union of nested Cantor sets:
	\begin{equation*}
	\mathcal{E}_f = \bigcup_{n=1}^\infty \mathcal{C}_n, \quad \mathcal{C}_n := \bigcap_{k=0}^\infty f^{-k} ([ r_n, 1 ]).
	\end{equation*}
	(The sequence $\{r_n\}$ is introduced in Definition \ref{def:rn}.)  The sets $\mathcal{C}_n$ are $f$-invariant, hence $\tilde{f}$-invariant, where $\tilde{f}$ is any expanding $C^2$ circle map that agrees with $f$ on $[r_n,1]$.  That $\Leb ( \mathcal{C}_n ) = 0$ can be deduced from the fact that $\tilde{f}$ preserves an ergodic invariant measure equivalent to Lebesgue
	measure \cite[Chapter 3, Theorem III.1.1]{MR889254}.  A more elementary proof that $\Leb ( \mathcal{C}_n ) = 0$ involves a simple distortion estimate, but the author is unaware of a published reference for this approach.
\end{remark}

One consequence of Theorem \ref{thm:maintheoremf} concerns the set of points having positive lower Lyapunov exponent for $f$.  Recall that for $x \in [ 0,1 ]$ the lower Lyapunov exponent of $x$ is the number
\begin{equation*}
\underbar{L}(x) := \liminf_{n \to \infty} \bigg( \frac{1}{n} \log \big\lvert ( f^n )' ( x ) \big\rvert \bigg)
\end{equation*}
(using one-sided derivatives as necessary).  Now define
\begin{equation*}
S := \{ x \in [ 0,1 ] \colon \underbar{L}(x) > 0 \}.
\end{equation*}
It is easy to see that $\mathcal{E}_f \subset S$.  Indeed, if $x \in \mathcal{E}_f$, find $\epsilon > 0$ such that the orbit of $x$ under $f$ avoids $[ 0, \epsilon )$.  Then
\begin{equation*}
\underbar{L}(x) \geq \liminf_{n \to \infty} \Biggl( \frac{1}{n} \sum_{k=0}^{n-1} \log \big\lvert f' \big( f^k x \big) \big\rvert \Biggr) \geq \log \Big( \inf_{x \in [\epsilon,1]} \big\lvert f'(x) \big\rvert \Big) > 0.
\end{equation*}
The following result is now immediate.

\begin{cor} \label{cor:Lyap}
	The set of points with positive lower Lyapunov exponent for $f$ is strong winning for Schmidt's game.
\end{cor}

It was known \cite{MR2505322,MR1764931} that $S$ has full Hausdorff dimension for all values of $\gamma > 0$; Corollary \ref{cor:Lyap} greatly strengthens this.  In the case that $\gamma < 1$, $f$ preserves a fully supported ergodic absolutely continuous (with respect to Lebesgue) measure $\mu$, so that Lebesgue-almost every point has positive lower Lyapunov exponent since
\begin{equation*}
\Lyap(\mu) := \int_0^1 \log \, \lvert f' \rvert \, d\mu > 0
\end{equation*}
in this case (see \cite{MR599464} and references therein).  Note that even sets with full Lebesgue measure are not necessarily winning.  This is because the complement of a winning set is never winning by Theorem \ref{thm:schmidt} below, and there exist Lebesgue-null winning sets; an example is the set of reals not normal to a given base \cite{MR195595}.  When $\gamma \geq 1$, however, $\Leb(S) = 0$ \cite{MR599464}, and so Corollary \ref{cor:Lyap} is the strongest available result concerning the ``largeness'' of the set $S$ in this case, and gives another example of a Lebesgue-null winning set.

\begin{figure}
	\centering
	\begin{minipage}{0.49\textwidth}
		\centering
		\includegraphics[width=.85\textwidth]{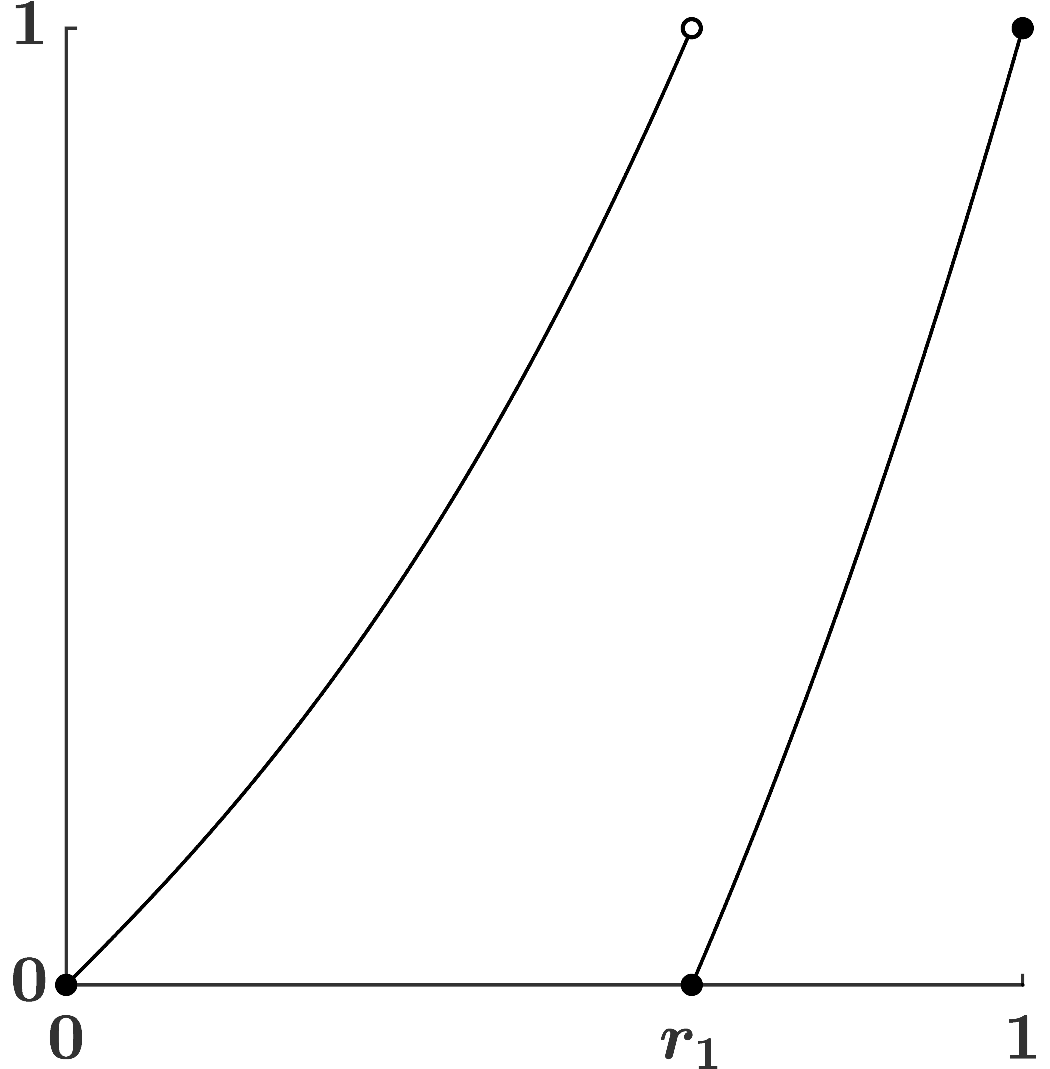}
		\captionsetup{width=.85\linewidth}
		\caption{The graph of the Manneville--Pomeau map $f$ with $\gamma = 1.5$.}
		\label{fig:MP}
	\end{minipage}\hfill
	\begin{minipage}{0.49\textwidth}
		\centering
		\includegraphics[width=.85\textwidth]{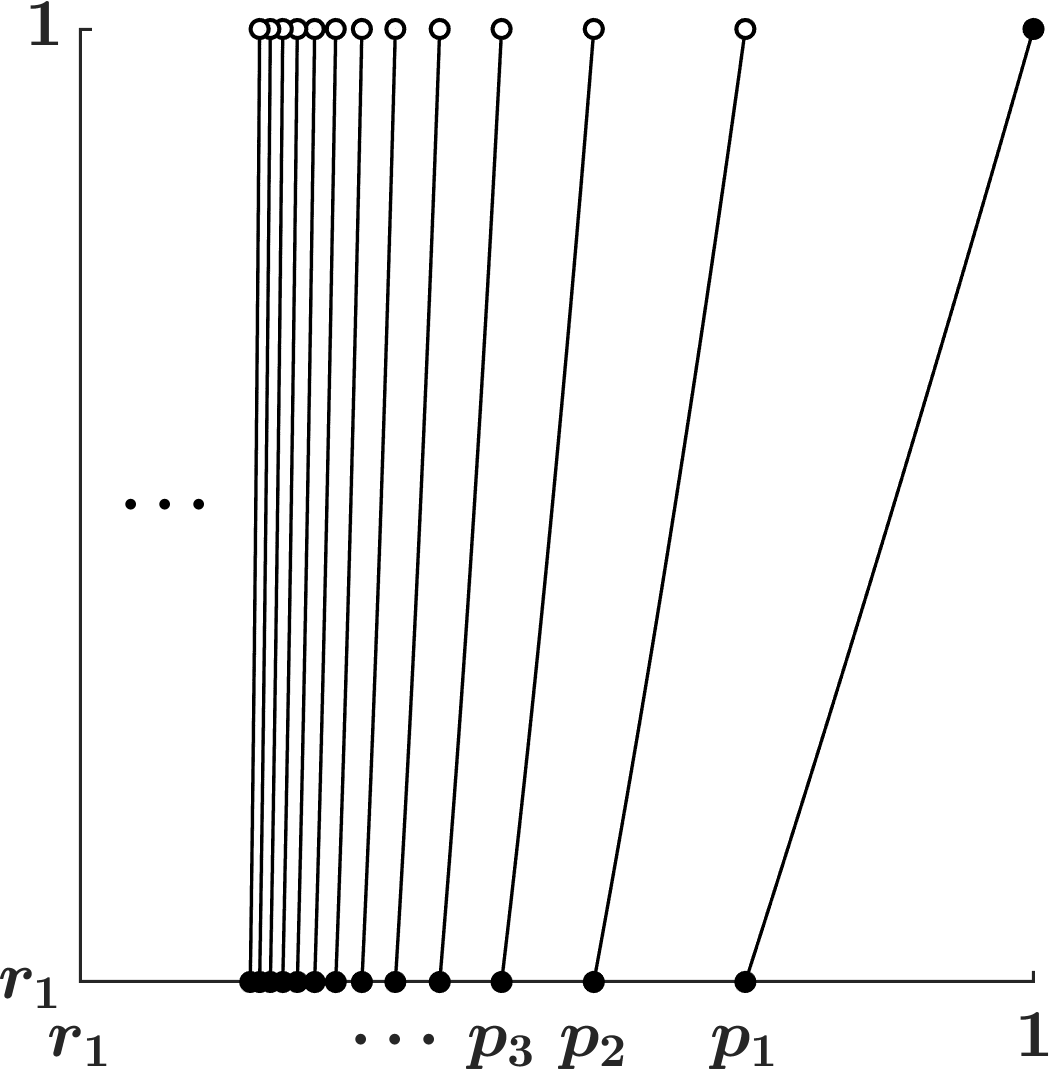}
		\captionsetup{width=.85\linewidth}
		\caption{The graph of $F$, the first return map to $[ r_1,1 ]$ induced by $f$ with $\gamma = 1.5$.}
		\label{fig:MPinduced}
	\end{minipage}
\end{figure}

\section{Method of proof}
The primary difficulty in studying $f$ is the nonuniformity of expansion near the indifferent fixed point 0, which gives rise to infinite distortion.  The map $f$ also exhibits unbounded geometry, by which we mean that the ratio of the longest to the shortest Markov partition element of successive generations tends to infinity.  We address the problem of infinite distortion by inducing $f$ on $[ r_1,1 ]$ to get a uniformly expanding first return map $F$.  This induced map satisfies a bounded distortion estimate, which is a key property of expanding systems that features prominently in the articles mentioned above \cite{MR2818688,MR980795,MR2660561,MR3206688,MR195595,MR2480100}.  The issue of unbounded geometry is handled using the notion of ``commensurate,'' introduced in \cite{MR3021798}.

The bulk of this paper involves analyzing the induced first return map
\begin{equation*}
	F \colon [ r_1, 1 ] \to [ r_1, 1 ]	
\end{equation*}
given by the rule
\begin{equation*}
F x := f^{\tau (x)} (x), \quad \tau (x) := \min \{ n \geq 1 \colon f^n x \in [ r_1,1 ] \}.
\end{equation*}
See Figure \ref{fig:MPinduced}.  We will show that Theorem \ref{thm:maintheoremf} is a straightforward consequence of the following analogous result for $F$, which we prove in \S\ref{sec:maintheoremF}:

\begin{thm} \label{thm:maintheoremF}
	With $F$ defined as above, the set
	\begin{equation*}
	\mathcal{E}_F := \{ x \in [ r_1,1 ] \colon [ r_1,r_1+\epsilon ) \cap \{ F^n x \}_{n \geq 0} = \emptyset \textrm{ for some } \epsilon > 0 \}
	\end{equation*}
	is strong winning for Schmidt's game.
\end{thm}

\begin{remark}
	Our proof of Theorem \ref{thm:maintheoremF} works for any map topologically conjugate to $F$ and satisfying the estimates concerning the Markov structure of $F$ in Proposition \ref{prop:C5}.
\end{remark}

In proving Theorem \ref{thm:maintheoremF} we follow the approach of Mance and Tseng in \cite{MR3021798}.  In that article the authors studied L\"uroth expansions, whose associated dynamical system is piecewise linear.  This linear structure permitted a precise computation of the lengths of intervals in the natural Markov partition.  In this paper we cannot obtain closed-form expressions for these lengths; instead we use estimates (Proposition \ref{prop:C5}) derived from a distortion result (Proposition \ref{prop:distF}).

We note that in \cite{MR3206688} Hu and Yu considered the class of piecewise locally $C^{1+\delta}$ expanding maps, a class that includes the Gauss map.  At first glance the induced map $F$ looks quite similar to the Gauss map; however, the authors in \cite{MR3206688} required a H\"older-type distortion estimate that $F$ does not satisfy.

\section{Schmidt's Game} \label{sec:Schmidt}
We describe a simplified version of a set-theoretic game introduced by Schmidt in \cite{MR195595}.  The game is played on the unit interval $[ 0,1 ]$.  Fix two constants $\alpha, \beta \in ( 0,1 )$ and a set $S \subset [ 0,1 ]$.  Two players, Alice and Bob, alternately choose nested closed intervals $B_1 \supset A_1 \supset B_2 \supset A_2 \supset \dots$ with Bob choosing first.  These intervals must satisfy the relations $\lvert B_{n+1} \rvert = \beta \lvert A_n \rvert$ and $\lvert A_n \rvert = \alpha \lvert B_n \rvert$ for all $n \in \mathbb{N}$ ($\lvert B_1 \rvert$ is arbitrary).  Then $\bigcap A_n = \bigcap B_n$ consists of a single point, $\omega$.  Alice wins the game if and only if $\omega \in S$.

If Alice has a winning strategy by which she can win regardless of Bob's choices, $S$ is said to be \textit{$( \alpha, \beta )$-winning}.  $S$ is called \textit{$\alpha$-winning} if it is $( \alpha, \beta )$-winning for all $\beta \in ( 0,1 )$.  $S$ is called \textit{winning} if it is $\alpha$-winning for some $\alpha \in ( 0,1 )$.  The following result lists important properties of winning sets; the proof may be found in \cite{MR195595}.

\begin{thm} \label{thm:schmidt}
	A winning set in $[ 0,1 ]$ is dense, uncountable, and has full Hausdorff dimension.  A countable intersection of $\alpha$-winning sets is $\alpha$-winning.  A co-countable subset of an $\alpha$-winning set is $\alpha$-winning.
\end{thm}

In \cite{MR2720230} McMullen introduced a modification of Schmidt's game in which the length restrictions are loosened to $\lvert B_{n+1} \rvert \geq \beta \lvert A_n \rvert$ and $\lvert A_n \rvert \geq \alpha \lvert B_n \rvert$.  This results in \textit{strong winning} sets.  It is easy to see that strong winning sets are winning.  McMullen showed that Theorem \ref{thm:schmidt} remains true when ``winning'' is replaced with ``strong winning'' and that the strong winning property is preserved under quasisymmetric homeomorphisms \cite[Theorem 1.1]{MR2720230}, which is not generally true of the winning property \cite[Theorem 1.2]{MR2720230}.

\section{Proofs of Preliminary Results} \label{sec:minorresults}
\subsection{Notation}
Let $B \subset [ 0,1 ]$ be a closed interval.  The expression $[B)$ denotes the interior of $B$ union its left endpoint; $(B]$ is similarly defined.  The left and right endpoints of $B$ are denoted by $\partial^\ell B$ and $\partial^r B$, respectively.  The notations $\overline{B}$ and $B^\circ$ denote the closure and interior of $B$, respectively.  The diameter of $B$ is represented by $\lvert B \rvert$, and we call $B$ \textit{nontrivial} if $0 < \lvert B \rvert < 1$.  Henceforth ``closed interval'' shall be understood to mean ``nontrivial closed interval''.

\subsection{Technical results}
\begin{defn}[The sequence $\{ r_n \}_{n=0}^\infty$] \label{def:rn}
	Define $\{ r_n \}_{n=0}^\infty \subset ( 0,1 ]$ recursively by $r_0 = 1$ and $\{ r_{n+1} \} = f^{-1} ( r_n ) \cap ( 0,r_n )$; thus $r_n \searrow 0$.
\end{defn}

\begin{defn}[The sequence $\{ p_n \}_{n=0}^\infty$] \label{def:pn}
	Define $\{ p_n \}_{n=0}^\infty \subset ( r_1,1 ]$ by $p_0 = 1$ and $\{ p_n \} := f^{-1} ( r_n ) \cap ( r_1,1 )$ for $n \geq 1$; thus $p_n \searrow r_1$.
\end{defn}

The asymptotics of these sequences will play a crucial role.  Proofs of the next two results may be found in \S6.2 of \cite{MR1750438}.  Note that in \S6 of that paper, Young imposes slightly more restrictive conditions on $f$ (continuity of $f''$ at $r_1$); however, her proofs of the following two theorems hold in the present setting.

\begin{thm}[The asymptotics of $\{ r_n \}_{n=0}^\infty$] \label{thm:rnyoung}
	There exists a constant $C_1 > 1$ such that for all $n \in \mathbb{N}$,
	\begin{gather*}
	C_1^{-1} n^{-1/\gamma} \leq r_n \leq C_1 n^{-1/\gamma}, \\
	C_1^{-1} n^{-1-1/\gamma} \leq r_{n-1}-r_n \leq C_1 n^{-1-1/\gamma}.
	\end{gather*}
\end{thm}

\begin{thm}[{A distortion estimate for $f^n \restriction_{[ 0,r_1 ]}$}] \label{thm:youngdist}
	There exists a constant $C_2 > 1$ such that for all integers $1 \leq m \leq n$, and for all points $x,y \in [ r_{n+1}, r_n)$,
	\begin{equation*}
	\bigg\lvert \log \frac{( f^m )'x}{( f^m )'y} \bigg\rvert \leq \frac{C_2}{r_{n-m} - r_{n-m+1}} \big\lvert f^m x - f^m y \big\rvert.
	\end{equation*}
\end{thm}

\begin{cor}[{A distortion estimate for $f^n \restriction_{[ r_1,1 ]}$}] \label{cor:C3}
	There exists a constant $C_3 > 1$ such that for all integers $1 \leq m \leq n$, and for all points $x,y \in [ p_n, p_{n-1} )$,
	\begin{equation*}
	\bigg\lvert \log \frac{( f^m )'x}{( f^m )'y} \bigg\rvert \leq \frac{C_3}{r_{n-m} - r_{n-m+1}} \big\lvert f^m x - f^m y \big\rvert.
	\end{equation*}
\end{cor}
\begin{proof}
	First assume that $m > 1$.  Observe that
	\begin{equation*}
	\bigg\lvert \log \frac{( f^m )'x}{( f^m )'y} \bigg\rvert \leq \bigg\lvert \log \frac{( f^{m-1} )' ( fx )}{( f^{m-1} )' ( fy )} \bigg\rvert + \big\lvert ( \log f' ) x - ( \log f' ) y \big\rvert.
	\end{equation*}
	Because $fx,fy \in [ r_n, r_{n-1})$, Theorem \ref{thm:youngdist} applies to the first term on the right-hand side above.  Now use the mean-value theorem to find $\xi \in ( x,y )$ such that
	\begin{align*}
	\bigg\lvert \log \frac{( f^m )'x}{( f^m )'y} \bigg\rvert &\leq \frac{C_2}{r_{n-m} - r_{n-m+1}} \big\lvert f^m x - f^m y \big\rvert + \bigg\lvert \frac{f'' \xi}{f' \xi} \bigg\rvert \lvert x - y \rvert \\
	&\leq \frac{C_2}{r_{n-m} - r_{n-m+1}} \big\lvert f^m x - f^m y \big\rvert + \sup \lvert f'' \rvert \cdot \big\lvert f^m x - f^m y \big\rvert \\
	&\leq \frac{C_2 + \sup \lvert f'' \rvert}{r_{n-m} - r_{n-m+1}} \big\lvert f^m x - f^m y \big\rvert.
	\end{align*}
	If $m=1$, then as above we have
	\begin{align*}
	\bigg\lvert \log \frac{( f^m )'x}{( f^m )'y} \bigg\rvert &= \big\lvert (\log f')x - (\log f')y \big\rvert \leq \sup \lvert f'' \rvert \cdot \lvert fx - fy \rvert \\
	&\leq \frac{C_2 + \sup \lvert f'' \rvert}{r_{n-m} - r_{n-m+1}} \big\lvert f^mx - f^my \big\rvert.
	\end{align*}
	The corollary follows by taking $C_3 := C_2 + \sup \lvert f'' \rvert$.
\end{proof}

\begin{defn}[Basic intervals of generation $n$; $G_n$] \label{def:basicinterval}
	Define the \textit{basic interval of generation $0$} to be $[ r_1, 1 ]$ and write $G_0 := \{ [ r_1, 1 ] \}$.  For $n \in \mathbb{N}$, a closed interval is called a \textit{basic interval of generation $n$} if it is the closure of a maximal open interval of monotonicity for $F^n$.  We denote by $G_n$ the collection of all basic intervals of generation $n$.  Thus, for example, $G_1 = \{ [ p_1,1 ], [ p_2,p_1 ], \ldots \}$.
\end{defn}

\begin{defn}[Labeling basic intervals via their itineraries] \label{def:Jsigma}
	Given $k \in \mathbb{N}$ and positive integers $m_1, \dots, m_k$, define $J_{m_1 \dots m_k} \in G_k$ as
	\begin{equation*}
	J_{m_1 \dots m_k} := \overline{{\textstyle \bigcap_{i=1}^k F^{-( i - 1 )} ( [ p_{m_i}, p_{m_i - 1}) ) }}.
	\end{equation*}
	Equivalently, we may recursively define $J_1 := [ p_1, 1 ]$, $J_2 := [ p_2, p_1 ]$, etc., and then declare $J_{m_1 \dots m_k} := J_{m_1 \dots m_{k-1}} \cap F^{-(k-1)} ( [ J_{m_k} ) )$.  Thus $J_{m_1 \dots m_k}$ is the $m_k$-th basic interval of $G_k$ in $J_{m_1 \dots m_{k-1}}$, with basic intervals numbered from right to left.
\end{defn}

In the following proposition we use the fact that $F$ is uniformly expanding, which follows from assumption (\ref{cond3}) in our definition of $f$. Define
\begin{equation*}
\lambda := \inf \big\{ \big\lvert F'x \big\rvert \colon x \in ( r_1, 1 ) \setminus \{ p_n \}_{n=1}^\infty \} > 1.
\end{equation*}

\begin{prop}[A distortion estimate for $F^n$] \label{prop:distF}
	There exists a constant $C_4 > 1$ such that for all integers $1 \leq k \leq n$, for all $J_{m_1 \dots m_n} \in G_n$, and for all $x,y \in ( J_{m_1 \dots m_n} )^\circ$,
	\begin{equation*}
	\bigg\lvert \log \frac{( F^k )'x}{( F^k )'y} \bigg\rvert \leq C_4.
	\end{equation*}
\end{prop}
\begin{proof}
	Because $F^{i-1} x, F^{i-1} y \in ( p_{m_i}, p_{m_i-1} )$ for $1 \leq i \leq n$, we have
	\begin{equation*}
	\bigg\lvert \log \frac{( F^k )'x}{( F^k )'y} \bigg\rvert \leq \sum_{i=1}^k \bigg\lvert \log \frac{F' ( F^{i-1} x )}{F' ( F^{i-1} y )} \bigg\rvert = \sum_{i=1}^k \bigg\lvert \log \frac{( f^{m_i} )' ( F^{i-1} x )}{( f^{m_i} )' ( F^{i-1} y )} \bigg\rvert.
	\end{equation*}
	Now we use Corollary \ref{cor:C3} to obtain
	\begin{align*}
	\bigg\lvert \log \frac{( F^k )'x}{( F^k )'y} \bigg\rvert &\leq \frac{C_3}{r_0-r_1} \sum_{i=1}^k \big\lvert f^{m_i} ( F^{i-1} x ) - f^{m_i} ( F^{i-1} y ) \big\rvert \\
	&= \frac{C_3}{r_0-r_1} \sum_{i=1}^k \big\lvert F^i x - F^i y \big\rvert \\
	&\leq \frac{C_3}{r_0-r_1} \sum_{i=1}^k \lambda^{-( k - i )} \big\lvert F^k x - F^k y \big\rvert \\
	&< C_3 \sum_{j=0}^\infty \lambda^{-j} =: C_4. \qedhere
	\end{align*}
\end{proof}

\begin{prop}[An estimate of the lengths of basic intervals] \label{prop:C5}
	There exists a constant $C_5 \geq 1$ such that for all $n \in \mathbb{N}$, for all $J_\sigma \in G_n$, and for all $k \in \mathbb{N}$,
	\begin{gather*}
	C_5^{-1} k^{-1/\gamma} \leq \frac{\big\lvert \bigcup_{i=k}^\infty J_{\sigma i} \big\rvert}{\lvert J_\sigma \rvert} \leq C_5 k^{-1/\gamma}, \\
	C_5^{-1} k^{-1-1/\gamma} \leq \frac{\lvert J_{\sigma k} \rvert}{\lvert J_\sigma \rvert} \leq C_5 k^{-1-1/\gamma}.
	\end{gather*}
\end{prop}
\begin{proof}
	The first claimed estimate is trivially satisfied when $k=1$ as long as $C_5 \geq 1$.  For $k > 1$, use the mean-value theorem to find $\xi_1, \xi_2 \in ( J_\sigma )^\circ$ such that
	\begin{equation*}
	\frac{\big\lvert \bigcup_{i=k}^\infty J_{\sigma i} \big\rvert}{\lvert J_\sigma \rvert} = \frac{\lvert [ r_1,p_{k-1} ] \rvert / ( F^n )' (\xi_1)}{\lvert [ r_1,1 ] \rvert / ( F^n )' ( \xi_2)}.
	\end{equation*}
	Now using Proposition \ref{prop:distF} and the first estimate of Theorem \ref{thm:rnyoung} yields
	\begin{align*}
	\frac{\big\lvert \bigcup_{i=k}^\infty J_{\sigma i} \big\rvert}{\lvert J_\sigma \rvert} &\leq \frac{\exp(C_4)}{1-r_1} \lvert [ r_1,p_{k-1} ] \rvert \leq \frac{\exp( C_4 )}{1-r_1} \lvert f ( [ r_1,p_{k-1} ] ) \rvert \\
	&= \frac{\exp( C_4 )}{1-r_1} r_{k-1} \leq \frac{C_1 \exp( C_4 )}{1-r_1} ( k-1 )^{-1/\gamma} \\
	&\leq \frac{2^{1/\gamma} C_1 \exp( C_4 )}{1-r_1} k^{-1/\gamma}.
	\end{align*}
	Similarly we have
	\begin{align*}
	\frac{\big| \bigcup_{i=k}^\infty J_{\sigma i} \big|}{| J_\sigma |} &\geq \frac{\exp( -C_4 )}{1-r_1} | [ r_1,p_{k-1} ] | \geq \frac{\exp( -C_4 )}{( 1-r_1 ) \sup \lvert f' \rvert} | f ( [ r_1,p_{k-1} ] ) | \\
	&\geq \frac{C_1^{-1} \exp( -C_4 )}{( 1-r_1 ) \sup \lvert f' \rvert} ( k-1 )^{-1/\gamma} \\
	&\geq \frac{C_1^{-1} \exp( -C_4 )}{( 1-r_1 ) \sup \lvert f' \rvert} k^{-1/\gamma}.
	\end{align*}
	For the second claimed estimate, use the second inequality of Theorem \ref{thm:rnyoung} and as above we find
	\begin{align*}
	\frac{| J_{\sigma k} |}{| J_\sigma |} &\leq \frac{\exp( C_4 )}{1-r_1} | [ p_k,p_{k-1} ] | \leq \frac{\exp( C_4 )}{1-r_1} | f ( [ p_k,p_{k-1} ] ) | \\
	&= \frac{\exp( C_4 )}{1-r_1} | [ r_k, r_{k-1} ] | \leq \frac{C_1 \exp( C_4 )}{1-r_1} k^{-1-1/\gamma}
	\end{align*}
	as well as
	\begin{align*}
	\frac{| J_{\sigma k} |}{| J_\sigma |} &\geq \frac{\exp( -C_4 )}{1-r_1} | [ p_k,p_{k-1} ] | \geq \frac{\exp( -C_4 )}{( 1-r_1 ) \sup \lvert f' \rvert} | f ( [ p_k,p_{k-1} ] ) | \\
	&= \frac{\exp( -C_4 )}{( 1-r_1 ) \sup \lvert f' \rvert} | [ r_k,r_{k-1} ] | \geq \frac{C_1^{-1} \exp( -C_4 )}{( 1-r_1 ) \sup \lvert f' \rvert} k^{-1-1/\gamma}.
	\end{align*}
	The proposition follows by taking
	\begin{equation*}
	C_5 := \max \bigg\{ 1, \frac{2^{1/\gamma} C_1 \exp( C_4 )}{1-r_1}, \frac{( 1-r_1 ) \sup \lvert f' \rvert}{C_1^{-1} \exp( -C_4 )} \bigg\}. \qedhere
	\end{equation*}
\end{proof}

\begin{cor}\label{cor:kfromzeta}
	Fix $n \in \mathbb{N}$, $J_\sigma \in G_n$, and $\zeta \in ( 0,1 )$.  Find the unique $K \in \mathbb{N}$ such that $\partial^\ell J_\sigma + \zeta | J_\sigma | \in [ J_{\sigma K} )$.  Then
	\begin{equation*}
	( C_5 \zeta )^{-\gamma}  - 1 \leq K \leq \big( C_5^{-1} \zeta \big)^{-\gamma}.
	\end{equation*}
\end{cor}
\begin{proof}
	Because
	\begin{equation*}
	\bigcup_{i=K+1}^\infty J_{\sigma i} \subset \big( \partial^\ell J_\sigma, \partial^\ell J_\sigma + \zeta | J_\sigma | \big] \subset \bigcup_{i=K}^\infty J_{\sigma i},
	\end{equation*}
	Proposition \ref{prop:C5} allows us to estimate the diameters of the three sets above as follows:
	\begin{equation*}
	C_5^{-1} ( K+1 )^{-1/\gamma} | J_\sigma | \leq \Bigg| \bigcup_{i=K+1}^\infty J_{\sigma i} \Bigg| \leq \zeta | J_\sigma | \leq \Bigg| \bigcup_{i=K}^\infty J_{\sigma i} \Bigg| \leq C_5 K^{-1/\gamma} | J_\sigma |.
	\end{equation*}
	Solving the inequalities
	\begin{equation*}
	C_5^{-1} ( K+1 )^{-1/\gamma} \leq \zeta \quad \textrm{and} \quad \zeta \leq C_5 K^{-1/\gamma}
	\end{equation*}
	for $K$ completes the proof.
\end{proof}

\section{Commensurability} \label{sec:commensurability}
Throughout this section, all closed intervals $B \subset [r_1,1]$ are assumed to be nontrivial: $0 < |B| < |[r_1,1]|$. Following \cite{MR3021798}, we make the next two definitions.

\begin{defn}[Left endpoints of generation $n$] \label{def:leftend}
	A point is called a \textit{left endpoint of generation $n$} if it is the left endpoint of some basic interval of generation $n$.
\end{defn}

\begin{defn}[Commensurability with generation $n$] \label{def:commensurability}
	Given a closed interval $B$ and $n \in \mathbb{N}$, $B$ is \textit{commensurate with generation $n$ (c.w.g.\ $n$)} if $B$ contains some member of $G_n$ but no member of $G_{n-1}$.
\end{defn}

We observe the following properties of basic intervals:
\begin{enumerate}[label=(\roman*)]
	\item \label{obs1} For all $I \in G_n$ with $n > 0$, and all $0 \leq k \leq n-1$, there exists a unique member of $G_k$ properly containing $I$.
	\item \label{obs2} Basic intervals of distinct generations are either nested or disjoint.
	\item \label{obs3} Basic intervals of the same generation have disjoint interiors.
	\item \label{obs4} Every basic interval $I \in G_n$ has a unique left-adjacent basic interval in $G_n$.
	\item \label{obs5} Every basic interval $I_{\sigma k} \in G_n$, where $|\sigma| \geq 0$ and $k > 1$, has a unique right-adjacent basic interval in $G_n$.
	\item \label{obs6} If $\ell$ is a left endpoint of generation $n$ and $\epsilon > 0$, then the interval $( \ell, \ell+\epsilon )$ contains infinitely many members of $G_{n+k}$ for all $k \geq 1$.
	\item \label{obs7} For each $n \in \mathbb{N} \cup \{ 0 \}$, the union of the elements of $G_n$ is co-countable, hence dense, in $[ r_1,1 ]$.
\end{enumerate}

\begin{lem}
Every closed interval $B \subset [r_1,1]$ is commensurate with a unique generation.
\end{lem}
\begin{proof}
	Since $F$ is uniformly expanding, the mesh of $G_n$ tends to zero.  Together with Observation \ref{obs7}, this implies that $B$ eventually intersects at least two adjacent basic intervals of a common generation, and hence contains a left endpoint of a basic interval. Hence $B$ contains some basic interval by Observation \ref{obs6}.  Let $n_0$ be the least generation for which $B$ contains a member of $G_{n_0}$.  Then $n_0 \geq 1$, and $B$ contains a member of $G_{n_0}$ but no member of $G_{n_0-1}$.

	Suppose $B$ is c.w.g.\ $g_1$ and $g_2$, where $g_1 < g_2$.  $B$ contains some $I \in G_{g_1}$; hence $[ B )$ contains $\partial^\ell I$.  Thus $\big( \partial^\ell I, \partial^r B \big) \subset B$ contains an element of $G_{g_1+1}$ by Observation \ref{obs6}.  Repeating this argument shows that $B$ contains an element of $G_{g_2 - 1}$, contradicting that $B$ is c.w.g.\ $g_2$.
\end{proof}

\begin{cor} \label{cor:atmost2}
	If a closed interval $B \subset [r_1,1]$ is c.w.g.\ $n$, then $B$ intersects either one or two elements of $G_{n-1}$.
\end{cor}
\begin{proof}
	$B$ intersects at least one member of $G_{n-1}$ by Observation \ref{obs7}.  If $B$ intersects three elements of $G_{n-1}$, then $B$ intersects three adjacent elements of $G_{n-1}$, hence properly contains the middle one, contradicting that $B$ is c.w.g.\ $n$.
\end{proof}

\begin{lem} \label{lem:accum}
	If a closed interval $B \subset [r_1,1]$ is c.w.g.\ $n$, then $B$ contains at most one left endpoint of generation at most $n-1$.  Furthermore, if $B$ contains a left endpoint $\ell$ of generation $k < n-1$, then $\ell$ is the right endpoint of $B$.
\end{lem}
\begin{proof}
	Suppose $B$ contains two left endpoints $\ell_1 < \ell_2$ of generations $g_1, g_2$, respectively, and $g_1, g_2 \leq n-1$.  First assume that $g_1 = g_2$.  Then $B$ contains two adjacent left endpoints of generation $g_1$; hence $B$ contains a basic interval of generation $g_1 \leq n-1$, contradicting that $B$ is c.w.g.\ $n$.
	
	Next assume $g_1 < g_2$.  Then the interval $( \ell_1,\ell_2 )$ contains an element of $G_{g_1+1}$ by Observation \ref{obs6}; hence $( \ell_1,\ell_2 )$ contains a left endpoint of generation $g_1+1$.  Repeating this argument shows that $( \ell_1,\ell_2 ) \subset B$ contains a left endpoint of generation $g_2$.  Now we are in the situation of the previous case, giving a contradiction.
	
	Finally, assume $g_1 > g_2$.  For $i \in \{ 1,2 \}$ let $I_i$ be the basic interval of generation $g_i$ with left endpoint $\ell_i$.  By Observation \ref{obs2}, either $I_1 \cap I_2 = \emptyset$, $I_1 \subset I_2$, or $I_2 \subset I_1$.  Now $I_2 \subset I_1$ is impossible because $g_2 < g_1$, and $I_1 \subset I_2$ is impossible because $\partial^\ell I_1 \notin I_2$.  So $I_1 \cap I_2 = \emptyset$ and thus $B$ contains $I_1$, a basic interval of generation at most $n-1$.  This contradicts that $B$ is c.w.g.\ $n$.
	
	For the second claim of the lemma, observe that if $[ B )$ contains a left endpoint $\ell$ of generation $k < n-1$, then the interval $(\ell, \partial^r B) \subset B$ contains a basic interval of generation $k+1 < n$ by Observation \ref{obs6}, contradicting that $B$ is c.w.g.\ $n$.
\end{proof}

\begin{cor} \label{cor:gnminus2}
	If a closed interval $B \subset [r_1,1]$ is c.w.g.\ $n \geq 2$, then there is a unique element of $G_{n-2}$ that properly contains $B$.
\end{cor}
\begin{proof}
	The interval $[ B )$ intersects at least one member of $G_{n-2}$ by Observation \ref{obs7}.  If $B$ intersects two members of $G_{n-2}$, then $B$ intersects two adjacent members $I_1 < I_2$ of $G_{n-2}$.  By Lemma \ref{lem:accum}, $\partial^\ell I_2 = \partial^r B$.  This shows that there is exactly one element of $G_{n-2}$ that intersects $[ B )$; hence this element must contain $B$ by Observation \ref{obs4}.  Proper containment follows because $B$ is c.w.g.\ $n$.
\end{proof}

\section{Proof that $\mathcal{E}_F$ is strong winning (Theorem \ref{thm:maintheoremF})} \label{sec:maintheoremF}
\subsection{Initial steps}
Recall the constant $C_5 \geq 1$ provided by Proposition \ref{prop:C5}, in which bounds on the lengths of basic intervals are derived; $\gamma > 0$, which appears in the exponent in the definition of $f$, controls the degree of nonuniform hyperbolicity of the system.  Define
\begin{equation*}
	\alpha := 2^{-2-1/\gamma} C_5^{-1} \in \big( 0,\tfrac{1}{4} \big)
\end{equation*}
and let $\beta \in ( 0,1 )$ be arbitrary.  We now show that $\mathcal{E}_F$ is $( \alpha, \beta )$-strong winning.

Bob begins the game by choosing $B_1 \subset [ r_1, 1 ]$.  Alice chooses $A_1 \subset B_1$ so that $\{ r_1, 1 \} \cap A_1 = \emptyset$.  Bob chooses $B_2 \subset A_1$.  Thus $B_2$ is c.w.g.\ $g_1 > 0$.

Find $d_1'$ large enough that $| B_2 | > (d_1')^{-1} | I |$ for all $I \in G_{g_1-1}$ that intersect $B_2$.  Next, if $g_1 = 1$, define $d_2' := 1$.  Otherwise find $d_2' > 1$ large enough so that $B_2 \cap \big( \partial^\ell I, \partial^\ell I + (d_2')^{-1} | I | \big) = \emptyset$ for all $I \in \bigcup_{g=0}^{g_1-2} G_g$.

Now fix constants $d_1$ and $d_2$ satisfying
\begin{align*}
d_1 &> \max \big\{ d_1', 2^{1+1/\gamma} C_5^4 ( \alpha \beta )^{-1} \big\}, \\
d_2 &> \max \big\{ d_2', 2^{1+2/\gamma} C_5^6 ( \alpha \beta )^{-1}, 2d_1 ( 1-2\alpha )^{-1} \big\}.
\end{align*}

Let $n_1 := 2$.  During the course of the $( \alpha, \beta )$ game we will prove the following claim, which is the heart of our proof, by induction.
\begin{claim}
	Regardless of how Bob plays the $( \alpha, \beta )$ game, Alice can play in such a way that: there exist integers $0 < n_1 < n_2 < \dots$ and $0 < g_1 < g_2 < \dots$ such that for all $j \in \mathbb{N}$, the following statements hold:
	\begin{enumerate}[label=$P_{\arabic*} ( j ):$,leftmargin=2\parindent,align=left]
		\item $B_{n_j}$ is c.w.g.\ $g_j$,
		\item $\big| B_{n_j} \big| > d_1^{-1} | I |$ for all $I \in G_{g_j-1}$ that intersect $B_{n_j}$,
		\item $B_{n_j} \cap \big( \partial^\ell I, \partial^\ell I + d_2^{-1} | I | \big) = \emptyset$ for all $I \in \bigcup_{g=0}^{g_j-2} G_g$.
	\end{enumerate}
\end{claim}
Note that the case $j=1$ was handled above.  Before proceeding to the induction step, we show how the claim implies the theorem.

Write $\{ \omega \} = \bigcap_{n=1}^\infty B_n$ and define $K := \lceil ( C_5 d_2 )^\gamma \rceil \geq 1$.  For any basic interval $J_\sigma$ of any generation we have
\begin{equation*}
\big( \partial^\ell J_\sigma, \partial^\ell J_\sigma + \tfrac{1}{d_2} | J_\sigma | \big) \supset \bigcup_{i=K+1}^\infty [ J_{\sigma i} )
\end{equation*}
by Proposition \ref{prop:C5}.  Also for any $n \in \mathbb{N} \cup \{ 0 \}$ we have $F^n \omega \in ( r_1, p_K )$ if and only if $\omega \in \bigcup_{i=K+1}^\infty J_{\sigma i}$ for some $J_\sigma \in G_n$.  The claim implies that the latter condition never holds; therefore the orbit of $\omega$ under $F$ stays outside $( r_1, p_K )$.  We conclude that
\begin{equation*}
\widetilde{\mathcal{E}_F} := \{ x \in [ r_1,1 ] \colon ( r_1, r_1+\epsilon ) \cap \{ F^n x \}_{n \geq 0} = \emptyset \textrm{ for some } \epsilon > 0 \}
\end{equation*}
is $( \alpha, \beta )$-strong winning.  As $\beta$ was arbitrary, $\widetilde{\mathcal{E}_F}$ is $\alpha$-strong winning.  Finally, the original set of interest, $\mathcal{E}_F$, is a co-countable subset of $\widetilde{\mathcal{E}_F}$ because
\begin{align*}
\mathcal{E}_F &= \{ x \in [ r_1,1 ] \colon [ r_1, r_1+\epsilon ) \cap \{ F^n x \}_{n \geq 0} = \emptyset \textrm{ for some } \epsilon > 0 \} \\
&= \widetilde{\mathcal{E}_F} \setminus \bigcup_{n=0}^\infty F^{-n} ( r_1 ).
\end{align*}
Therefore $\mathcal{E}_F$ is $\alpha$-strong winning (see Theorem \ref{thm:schmidt} and the paragraph following).

\subsection{Induction step of the claim} \label{subsec:inductionstep}
We will need the following result.
\begin{lem} \label{lem:1overbK}
	Fix a basic interval $J_\sigma$ of any generation.  Then
	\begin{equation*}
	\big[ \partial^\ell J_\sigma, \partial^\ell J_\sigma + \tfrac{1}{d_2} | J_\sigma | \big] \subset \big[ \partial^\ell J_\sigma, \partial^r J_{\sigma 3} \big).
	\end{equation*}
	Equivalently,
	\begin{equation*}
		\big[ \partial^\ell J_\sigma, \partial^\ell J_\sigma + \tfrac{1}{d_2} | J_\sigma | \big] \cap \big( J_{\sigma 1} \cup J_{\sigma 2} \big) = \emptyset.
	\end{equation*}
\end{lem}
\begin{proof}
	Let $K$ be the unique integer such that $\partial^\ell J_\sigma + d_2^{-1} | J_\sigma | \in [ J_{\sigma K} )$.  Using Corollary \ref{cor:kfromzeta} we find that
	\begin{equation*}
	K + 1 \geq \bigg( \frac{C_5}{d_2} \bigg)^{-\gamma} > \bigg( \frac{C_5}{2^{1+2/\gamma} C_5^6 ( \alpha \beta )^{-1}} \bigg)^{-\gamma} > \bigg( \frac{C_5}{3^{1/\gamma} C_5} \bigg)^{-\gamma} = 3.
	\end{equation*}
	Hence $K \geq 3$.
\end{proof}

Now we begin the induction.  Assume that for some $j \in \mathbb{N}$ statements $P_1 ( j )$, $P_2 ( j )$, and $P_3 ( j )$ hold.  By Lemma \ref{lem:accum}, $B_{n_j}$ contains at most one left endpoint of generation at most $g_j-1$.  Let $B_{n_j}^\textrm{mid}$ denote the midpoint of $B_{n_j}$.  We consider two cases, according as to whether the interval $\big. \big( B_{n_j}^\textrm{mid}, \partial^r B_{n_j} \big. \big]$ contains a left endpoint of generation at most $g_j-1$.

\subsection*{Case 1: The interval $\bm{\big( B_{n_j}^\textrm{mid}, \partial^r B_{n_j} \big]}$ does not contain a left endpoint of generation at most $\bm{g_j-1}$}
\begin{figure}
	\centering
	\includegraphics[width=\textwidth]{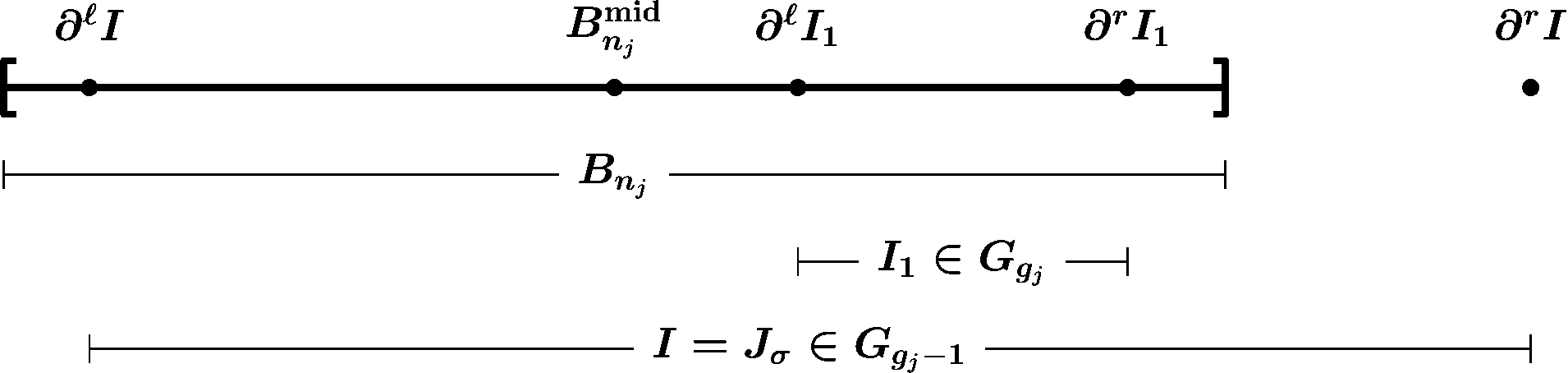}
	\caption{One possibility for Case 1 of the induction step.}
	\label{fig:case1}
\end{figure}

We refer the reader to Figure \ref{fig:case1}.  Because $B_{n_j}$ is c.w.g.\ $g_j$, $B_{n_j}$ contains some basic interval of generation $g_j$.  Let $I_1$ be the rightmost basic interval of generation $g_j$ contained in $B_{n_j}$, and let $I$ denote the unique basic interval of generation $g_j-1$ containing $I_1$ by Observation \ref{obs1}.  Then $\partial^\ell I \leq B_{n_j}^\textrm{mid}$.  Note that $\partial^\ell I$ could be inside or outside $B_{n_j}$.

Next, we claim that $\partial^r I > \partial^r B_{n_j}$.  To see this, first note that $\partial^r I \geq \partial^\ell B_{n_j}$ because $I_1$, and hence $I$, intersects $B_{n_j}$.  Next we have that $\partial^r I \geq \partial^r B_{n_j}$, for otherwise the interval $\big( \partial^r I, \partial^r B_{n_j} \big) \subset B_{n_j}$ would contain a member of $G_{g_j}$ to the right of $I_1$ by Observation \ref{obs6}.  Finally, if $\partial^r I = \partial^r B_{n_j}$, then $\partial^r B_{n_j} \leq \partial^r B_2 < 1$ and hence $\partial^r I \in \big. \big( B_{n_j}^\textrm{mid}, \partial^r B_{n_j} \big. \big]$ would be the left endpoint of some basic interval of generation at most $g_j-1$.  This proves the claim.

Write $I = J_\sigma$ for some string $\sigma$ of length $g_j-1$.  In order to specify Alice's strategy in choosing $A_{n_j}$ we consider two subcases, according as to whether $\partial^\ell J_{\sigma 1} \leq \partial^r B_{n_j}$.

\textit{Subcase 1: $\partial^\ell J_{\sigma 1} > \partial^r B_{n_j}$.} See Figure \ref{fig:case1subcase1}.  Alice chooses
\begin{equation*}
A_{n_j} = \big[ \partial^r B_{n_j} - \alpha \big| B_{n_j} \big|, \partial^r B_{n_j} \big] \subset B_{n_j}.
\end{equation*}
Using the induction hypothesis $P_2 ( j )$ we find that
\begin{align*}
\partial^r B_{n_j} - \bigg( \partial^\ell I + \frac{1}{d_2} | I | \bigg) &\geq \frac{1}{2} \big| B_{n_j} \big| - \frac{1}{d_2} | I | > \frac{1}{2} \big| B_{n_j} \big| - \frac{d_1}{d_2} \big| B_{n_j} \big| \\
&> \bigg( \frac{1}{2} - \frac{d_1}{2d_1 ( 1-2\alpha )^{-1}} \bigg) \big| B_{n_j} \big| = \alpha \big| B_{n_j} \big|.
\end{align*}
This shows that $A_{n_j}$ is disjoint from $\big( \partial^\ell I, \partial^\ell I + d_2^{-1} | I | \big)$.  Also $A_{n_j}$ is disjoint from $J_{\sigma 1}$ because $\partial^r A_{n_j} = \partial^r B_{n_j} < \partial^\ell J_{\sigma 1}$.  Finally, because $\alpha < \frac{1}{2}$ we have $A_{n_j} \subset \big[ B_{n_j}^\textrm{mid}, \partial^r B_{n_j} \big] \subset I$ so that $A_{n_j}$ is disjoint from every element of $G_{g_j-1} \setminus \{I\}$.

\begin{figure}
	\centering
	\includegraphics[width=\textwidth]{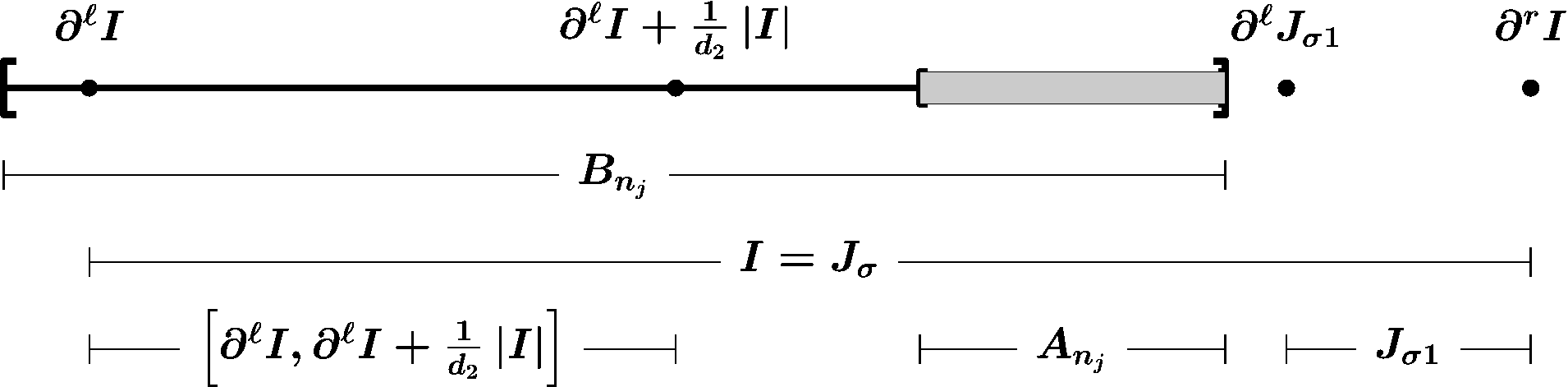}
	\caption{Case 1, Subcase 1 of the induction step.}
	\label{fig:case1subcase1}
\end{figure}

\textit{Subcase 2: $\partial^\ell J_{\sigma 1} \leq \partial^r B_{n_j}$.} See Figure \ref{fig:case1subcase2}.  In this case $B_{n_j}$ must contain $J_{\sigma 2}$ since otherwise $B_{n_j}$ would not contain any member of $G_{g_j}$.  Also $\big( \partial^\ell I, \partial^\ell I + d_2^{-1} | I | \big)$ is disjoint from $J_{\sigma 2}$ by Lemma \ref{lem:1overbK}.  Furthermore, by Proposition \ref{prop:C5},
\begin{align*}
| J_{\sigma 2} | &\geq 2^{-1-1/\gamma} C_5^{-1} | I | > 2^{-1-1/\gamma} C_5^{-1} \big| \big[ B_{n_j}^\textrm{mid}, \partial^r B_{n_j} \big] \big| \\
&= 2^{-2-1/\gamma} C_5^{-1} \big| B_{n_j} \big| = \alpha \big| B_{n_j} \big|.
\end{align*}
Thus, as in the previous subcase, Alice may choose $A_{n_j} \subset J_{\sigma 2} \subset I$ to be disjoint from $\big( \partial^\ell I, \partial^\ell I + d_2^{-1} |I| \big)$, $J_{\sigma 1}$, and every element of $G_{g_j-1} \setminus \{ I \}$.

\begin{figure}
	\centering
	\includegraphics[width=\textwidth]{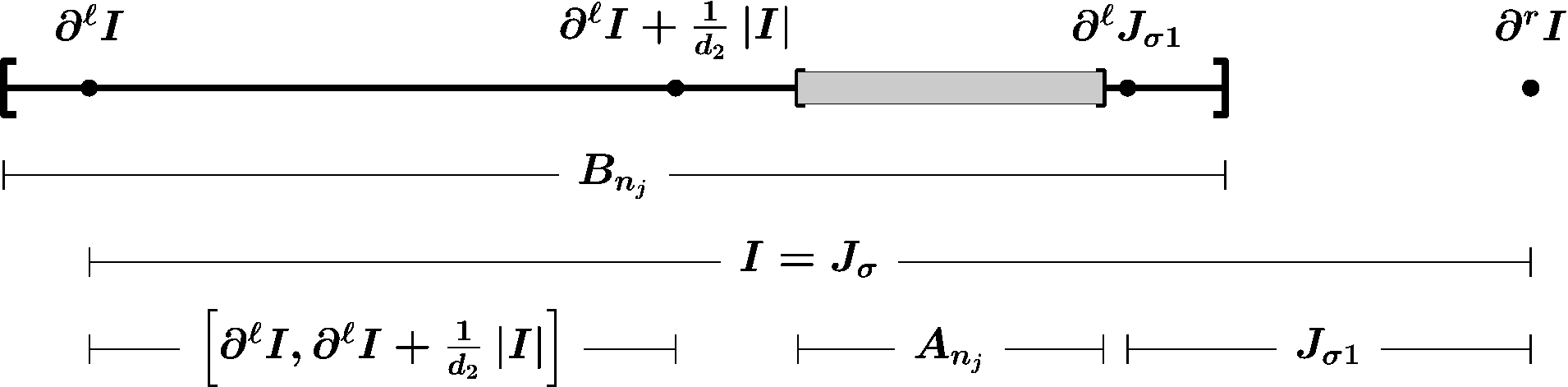}
	\caption{Case 1, Subcase 2 of the induction step.}
	\label{fig:case1subcase2}
\end{figure}

This takes care of the two subcases.  Now Bob chooses $B_{n_j+1}$.  If $B_{n_j+1}$ is c.w.g.\ $g_j$, Alice plays arbitrarily until Bob chooses an interval c.w.g.\ $g_{j+1} > g_j$.  This will eventually happen because $A_{n_j}$ contains finitely many members of $G_{g_j}$ (since $\partial^\ell I \notin A_{n_j}$) and Alice can force $| B_n | \searrow 0$ by always choosing an interval $A_n$ of length $\alpha | B_n |$; hence $B_n$ will eventually be too small to contain a member of $G_{g_j}$.

Let $n_{j+1}$ be such that $B_{n_{j+1}-1}$ is c.w.g.\ $g_j$ and $B_{n_{j+1}}$ is c.w.g.\ $g_{j+1} > g_j$.  Define
\begin{equation*}
\mathcal{J} := \Bigg\{ J \in \bigcup_{g=g_j}^{\mathclap{g_{j+1}-1}} G_g \colon J \cap B_{n_{j+1}} \neq \emptyset \Bigg\}.
\end{equation*}
Observe that every $J \in \mathcal{J}$ is contained in $J_\sigma$ because $B_{n_{j+1}}$ is disjoint from every element of $G_{g_j-1} \setminus \{ J_\sigma \}$.

\begin{lem} \label{lem:case1lem}
	We have
	\begin{equation*}
	\big| B_{n_{j+1}} \big| \geq 2^{-1-1/\gamma} \alpha \beta C_5^{-4} | J | > 2^{1/\gamma} C_5^2 d_2^{-1} | J |
	\end{equation*}
	for all $J \in \mathcal{J}$.
\end{lem}
\begin{proof}
	First observe that every $J \in \mathcal{J}$ is contained in some element of $G_{g_j} \cap \mathcal{J}$, and so it suffices to verify the lemma when $J \in G_{g_j} \cap \mathcal{J}$.  Because $B_{n_{j+1}-1} \subset I$ is c.w.g.\ $g_j$ we may define
	\[
	k_0 := \min \big\{ k \colon J_{\sigma k} \subset B_{n_{j+1}-1} \big\}.
	\]
	Then $k_0 \geq 2$ by the choice of $A_{n_j}$ (or because $\partial^r I \notin B_{n_j}$ if $B_{n_{j+1}-1} = B_{n_j}$).  By the definition of $k_0$ we have
	\[
	k_0-1 = \min \big\{ k \colon J_{\sigma k} \cap B_{n_{j+1}-1} \neq \emptyset \big\} \leq \min \{ k \colon J_{\sigma k} \in \mathcal{J} \}.
	\]
	Hence, if $J=J_{\sigma k}\in G_{g_j}\cap\mathcal{J}$, then $k\geq k_0-1$, and Proposition \ref{prop:C5} gives
	\[
	|J|=|J_{\sigma k}|
	\leq C_5 k^{-1-1/\gamma}|J_\sigma|
	\leq C_5 (k_0-1)^{-1-1/\gamma}|J_\sigma|
	\leq C_5^2 |J_{\sigma(k_0-1)}|.
	\]
	Using Proposition \ref{prop:C5} again, we have
	\begin{align*}
		\frac{ \big| B_{n_{j+1}} \big|}{|J|}
		&\geq \alpha \beta \frac{\big| B_{n_{j+1}-1} \big|}{C_5^2\big| J_{\sigma ( k_0-1 )} \big|}
		\geq \alpha \beta \frac{| J_{\sigma k_0} |}{C_5^2\big| J_{\sigma ( k_0-1 )} \big|} \\
		&\geq \alpha \beta \frac{C_5^{-1} k_0^{-1-1/\gamma}}{C_5^3 ( k_0-1 )^{-1-1/\gamma}}
		\geq 2^{-1-1/\gamma} \alpha \beta C_5^{-4} \\
		&= \frac{2^{1/\gamma} C_5^2}{2^{1+2/\gamma} C_5^6 ( \alpha \beta )^{-1}} > \frac{2^{1/\gamma} C_5^2}{d_2}. \qedhere
	\end{align*}
\end{proof}

\begin{cor} \label{cor:case1cor}
	The interval $B_{n_{j+1}}$ is disjoint from $\big( \partial^\ell J, \partial^\ell J + d_2^{-1} | J | \big)$ for all $J \in \bigcup_{g=0}^{g_{j+1}-2} G_g$.
\end{cor}
\begin{proof}
	$P_3 ( j )$ is true by the induction hypothesis; therefore it suffices to consider $J \in \bigcup_{g=g_j-1}^{g_{j+1}-2} G_g$.  Also $B_{n_{j+1}} \subset A_{n_j}$, $A_{n_j}$ is disjoint from $\big( \partial^\ell I, \partial^\ell I + d_2^{-1} | I | \big)$, and $I$ is the only element of $G_{g_j-1}$ that intersects $\big( A_{n_j} \big)^\circ$.  So it suffices to consider $J \in \bigcup_{g=g_j}^{g_{j+1}-2} G_g$.
	
	Fix such a $J = J_\tau \in G_{g'}$, where $g_j \leq g' \leq g_{j+1}-2$.  Using Observation \ref{obs1} and Corollary \ref{cor:gnminus2}, let $J'$ be the unique element of $G_{g'}$ containing $B_{n_{j+1}}$.  If $J \neq J'$, then $B_{n_{j+1}}$ is disjoint from the interior of $J$ and we are done.  So suppose $J = J'$.
	
	Find the unique $K \in \mathbb{N}$ such that $\partial^\ell J + d_2^{-1} | J | \in [ J_{\tau K} )$.  By Lemma \ref{lem:1overbK}, $K-1 \geq 2$, and by Corollary \ref{cor:kfromzeta}, $K-1 \geq \big( C_5 d_2^{-1} \big)^{-\gamma} - 2$.  So by Proposition \ref{prop:C5},
	\begin{align*}
	\frac{\big| \bigcup_{i=K-1}^\infty J_{\tau i} \big|}{| J |} &\leq C_5 ( K-1 )^{-1/\gamma} \leq C_5 \bigg( \bigg( \frac{C_5}{d_2} \bigg)^{-\gamma}-2 \bigg)^{-1/\gamma}  \\
	&\leq C_5 \bigg( \frac{1}{2} \bigg( \frac{C_5}{d_2} \bigg)^{-\gamma} \bigg)^{-1/\gamma} = \frac{2^{1/\gamma} C_5^2}{d_2}.
	\end{align*}
	Therefore, if $\partial^\ell B_{n_{j+1}}$ were contained in $\big[ \partial^\ell J, \partial^\ell J + d_2^{-1} | J | \big)$, then $B_{n_{j+1}}$ would contain $J_{\tau ( K-1 )} \in G_{g'+1}$ by Lemma \ref{lem:case1lem}.  But this is not possible because $B_{n_{j+1}}$ is c.w.g.\ $g_{j+1} > g'+1$.
\end{proof}

In conclusion, $P_1 ( j+1 )$ is true by construction, Lemma \ref{lem:case1lem} implies $P_2 ( j+1 )$ because $d_1^{-1} < 2^{-1-1/\gamma} \alpha \beta C_5^{-4}$, and Corollary \ref{cor:case1cor} is the statement $P_3 ( j+1 )$.  This completes the analysis of Case 1.

\subsection*{Case 2: The interval $\bm{\big( B_{n_j}^\textrm{mid}, \partial^r B_{n_j} \big]}$ contains a left endpoint of generation at most $\bm{g_j-1}$}
\begin{figure}
	\centering
	\includegraphics[width=\textwidth]{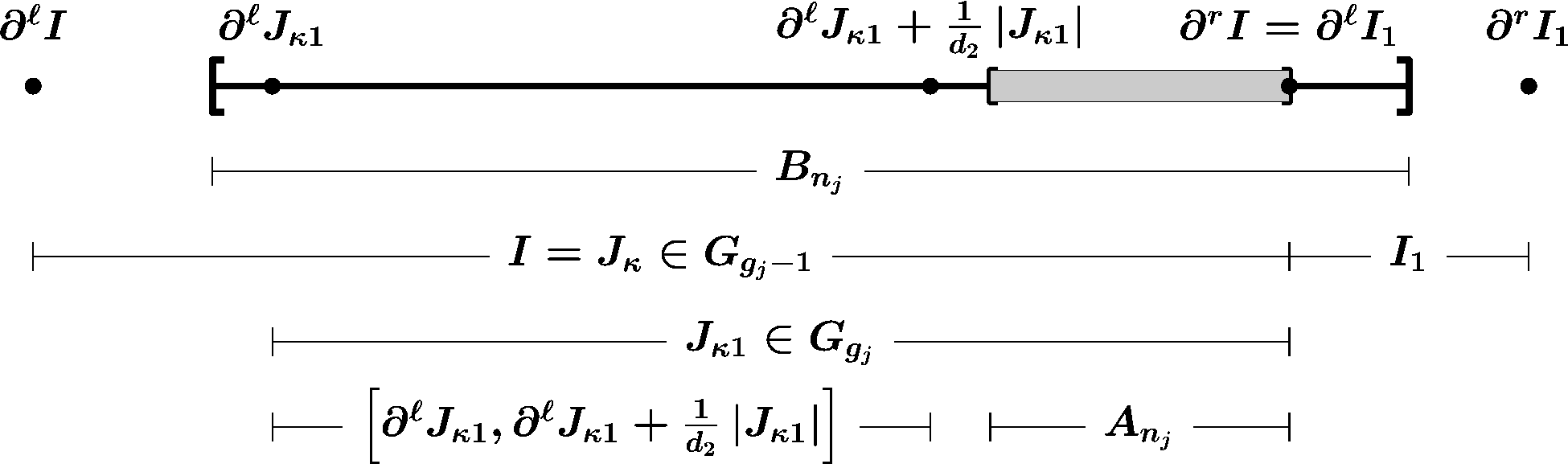}
	\caption{Case 2 of the induction step.}
	\label{fig:case2}
\end{figure}

We refer the reader to Figure \ref{fig:case2}.  Let $I_1$ be a basic interval of generation at most $g_j-1$ with left endpoint in $\big. \big( B_{n_j}^\textrm{mid}, \partial^r B_{n_j} \big. \big]$.  Then there is some basic interval of generation at most $g_j-1$ with right endpoint $\partial^\ell I_1$ by Observation \ref{obs4}; hence there is some $I = J_{\kappa} \in G_{g_j-1}$ having right endpoint $\partial^\ell I_1$.  Note that $\partial^\ell I < \partial^\ell B_{n_j}$ since $\partial^r I \in B_{n_j}$ and $B_{n_j}$ is c.w.g.\ $g_j$.  Alice chooses $A_{n_j} = \big[ \partial^r I - \alpha \big| B_{n_j} \big|, \partial^r I \big]$.  Using Proposition \ref{prop:C5} we have
\begin{align*}
\bigg| \bigg[ \partial^\ell J_{\kappa 1} + \frac{1}{d_2} | J_{\kappa 1} |,\partial^r I \bigg] \bigg| &\geq C_5^{-1} | I | \bigg( 1- \frac{1}{d_2} \bigg) \\
&\geq C_5^{-1} \big| \big[ \partial^\ell B_{n_j}, B_{n_j}^\textrm{mid} \big] \big| \bigg( 1 - \frac{1}{d_2} \bigg) \\
&> \frac{1}{4} C_5^{-1} \big| B_{n_j} \big| > \alpha \big| B_{n_j} \big|,
\end{align*}
which shows that $A_{n_j} \subset J_{\kappa 1}$ and moreover, that $A_{n_j}$ is disjoint from the interval $\big[ \partial^\ell J_{\kappa 1}, \partial^\ell J_{\kappa 1} + d_2^{-1} | J_{\kappa 1} | \big]$.  Thus $A_{n_j}$ is disjoint from all intervals $\big[ \partial^\ell J, \partial^\ell J + d_2^{-1} | J | \big]$ where $J \in G_{g_j}$.

Let $A_{n_j}$ be c.w.g.\ $\tilde{g} > g_j$.  Then by the choice of $A_{n_j}$, $J_{\kappa 1 \tilde{\kappa} 1} \subset A_{n_j} \subset J_{\kappa 1 \tilde{\kappa}}$, where $\tilde{\kappa}$ is a string of $\tilde{g}-g_j-1$ repeating ones.  Now Bob chooses $B_{n_j+1}$.  Define $n_{j+1} := n_j + 1$ and let $B_{n_{j+1}}$ be c.w.g.\ $g_{j+1} \geq \tilde{g}$.

\begin{lem} \label{lem:case2lemsize}
	We have
	\begin{equation*}
	\big| B_{n_{j+1}} \big| \geq \beta C_5^{-1} | J | > \frac{1}{d_1} | J |
	\end{equation*}
	for all $J \in G_{g_{j+1}-1}$ that intersect $B_{n_{j+1}}$.
\end{lem}
\begin{proof}
	If $g_{j+1} = \tilde{g}$, then the only basic interval of generation $g_{j+1}-1$ intersecting $B_{n_{j+1}}$ is $J_{\kappa 1 \tilde{\kappa}}$, and by Proposition \ref{prop:C5} we have
	\begin{equation*}
	\big| B_{n_{j+1}} \big| \geq \beta \big| A_{n_j} \big| \geq \beta | J_{\kappa 1 \tilde{\kappa} 1} | \geq \beta C_5^{-1} | J_{\kappa 1 \tilde{\kappa}} | > \frac{1}{d_1} | J_{\kappa 1 \tilde{\kappa}} |.
	\end{equation*}
	On the other hand, if $g_{j+1} > \tilde{g}$, then there are at most two basic intervals of generation $g_{j+1}-1$ intersecting $B_{n_{j+1}}$ by Corollaries \ref{cor:atmost2} and \ref{cor:gnminus2}.  If there is one, call it $J_{\tau t}$; if there are two, call them $J_{\tau t}$ and $J_{\tau (t+1)}$.  Both $J_{\tau t}$ and $J_{\tau (t+1)}$ are contained in $J_{\kappa 1 \tilde{\kappa}}$.  Thus, borrowing from the calculation above,
	\begin{align*}
	\big| B_{n_{j+1}} \big| &\geq \beta C_5^{-1} | J_{\kappa 1 \tilde{\kappa}} | > \beta C_5^{-1} \max \{ | J_{\tau (t+1)} |, | J_{\tau t} | \} \\
	&> \frac{1}{d_1} \max \{ | J_{\tau (t+1)} |, | J_{\tau t} | \}. \qedhere
	\end{align*}
\end{proof}

\begin{lem} \label{lem:case2lem}
	The interval $B_{n_{j+1}}$ is disjoint from $\big( \partial^\ell J, \partial^\ell J + d_2^{-1} | J | \big)$ for all $J \in \bigcup_{g=0}^{g_{j+1}-2} G_g$.
\end{lem}
\begin{proof}
	$P_3 ( j )$ is true by the induction hypothesis; therefore it suffices to consider $J \in \bigcup_{g=g_j-1}^{g_{j+1}-2} G_g$.  Also $B_{n_{j+1}} \subset A_{n_j} \subset J_{\kappa 1}$, $J_{\kappa 1}$ is disjoint from $\big( \partial^\ell I, \partial^\ell I + d_2^{-1} | I | \big)$ by Lemma \ref{lem:1overbK}, and $I$ is the only element of $G_{g_j-1}$ that intersects $\big( A_{n_j} \big)^\circ$.  So it suffices to consider $J \in \bigcup_{g=g_j}^{g_{j+1}-2} G_g$.
	
	Fix such a $J \in G_g$, where $g_j \leq g \leq g_{j+1}-2$.  By Corollary \ref{cor:gnminus2}, $B_{n_{j+1}}$ is properly contained in a unique element of $G_{g_{j+1}-2}$.  By Observation \ref{obs1}, this element has a unique ancestor $J' \in G_g$.  If $J \neq J'$, then $B_{n_{j+1}}$ is disjoint from the interior of $J$ and we are done.  So suppose $J = J'$.
	
	Let $J^\ast := J_{\kappa 1 \tilde{\kappa}} \in G_{\tilde{g}-1}$.  We consider two cases, the first of which is potentially vacuous.
	
	\textit{Case A: $g_j \leq g \leq \tilde{g}-2$.}  Since $J$ contains $B_{n_{j+1}}$ and $B_{n_{j+1}} \subset J^\ast$, Observation \ref{obs1} implies that $J$ is the unique ancestor of $J^\ast$ in $G_g$.  Hence $J = J_{\kappa 1 \kappa'}$, where $\kappa'$ is a string of $g-g_j$ repeating ones.  Since
	\[
	| \kappa' | = g-g_j \leq \tilde{g}-g_j-2 < \tilde{g}-g_j-1 = | \tilde{\kappa} |,
	\]
	we have $J^\ast \subset J_{\kappa 1 \kappa' 1}$.  On the other hand, Lemma \ref{lem:1overbK} gives
	\[
	\big( \partial^\ell J, \partial^\ell J + d_2^{-1} | J | \big) \subset \bigcup_{i=3}^\infty J_{\kappa 1 \kappa' i}.
	\]
	Thus $B_{n_{j+1}} \subset J^\ast$ is disjoint from $\big( \partial^\ell J, \partial^\ell J + d_2^{-1} | J | \big)$.
	
	\textit{Case B: $\tilde{g}-1 \leq g \leq g_{j+1}-2$.}  Write $J = J_\sigma$.  Since $J$ contains $B_{n_{j+1}}$ and $B_{n_{j+1}} \subset J^\ast$, Observation \ref{obs2} implies that $J \subset J^\ast$ or $J=J^\ast$.  In either case, $|J| \leq |J^\ast|$.
	
	Find the unique $K$ such that
	\[
	\partial^\ell J + d_2^{-1} | J | \in [ J_{\sigma K} ).
	\]
	By Lemma \ref{lem:1overbK}, $K-1 \geq 2$, and by Corollary \ref{cor:kfromzeta}, $K-1 \geq \big( C_5 d_2^{-1}\big)^{-\gamma} - 2$.  Thus, using Proposition \ref{prop:C5},
	\begin{align*}
		\frac{\big| \bigcup_{i=K-1}^\infty J_{\sigma i} \big|}{| J |} 
		&\leq C_5 ( K-1 )^{-1/\gamma} 
		\leq C_5 \bigg( \bigg( \frac{C_5}{d_2} \bigg)^{-\gamma} -2 \bigg)^{-1/\gamma} \\
		&\leq C_5 \bigg( \frac{1}{2} \bigg( \frac{C_5}{d_2} \bigg)^{-\gamma} \bigg)^{-1/\gamma} 
		= \frac{2^{1/\gamma} C_5^2}{d_2}.
	\end{align*}
	Also, since $|J| \leq |J^\ast|$, Proposition \ref{prop:C5} gives
	\begin{equation*}
		\frac{\big| B_{n_{j+1}} \big|}{| J |} 
		\geq \beta \frac{\big| A_{n_j} \big|}{| J |} 
		\geq \beta \frac{| J_{\kappa 1 \tilde{\kappa} 1} |}{| J^\ast |} 
		\geq \beta C_5^{-1} 
		> \frac{2^{1/\gamma} C_5^2}{d_2}.
	\end{equation*}
	Therefore, if $\partial^\ell B_{n_{j+1}}$ were contained in $\big[ \partial^\ell J, \partial^\ell J + d_2^{-1} | J | \big)$, then $B_{n_{j+1}}$ would contain $J_{\sigma ( K-1 )} \in G_{g+1}$.  But this is not possible because $B_{n_{j+1}}$ is c.w.g.\ $g_{j+1} > g+1$.
\end{proof}

In conclusion, $P_1 ( j+1 )$ is true by construction, Lemma \ref{lem:case2lemsize} is the statement $P_2 ( j+1 )$, and Lemma \ref{lem:case2lem} is the statement $P_3 ( j+1 )$.  This completes the analysis of Case 2.  The induction argument is complete, and with it, the proof of Theorem \ref{thm:maintheoremF}.

\section{Proof that $\mathcal{E}_f$ is strong winning (Theorem \ref{thm:maintheoremf})} \label{sec:corollary}
Let $\mathcal{E}_F$ be $\alpha_F$-strong winning (so $0 < \alpha_F \leq \frac{1}{2}$). Recall the constant $C_2 > 1$ provided by Theorem \ref{thm:youngdist} and define
\begin{equation*}
	\alpha_f := \exp ( -C_2 ) \alpha_F \in \big( 0, \tfrac{1}{2} \big).
\end{equation*}
Let $\beta_f \in ( 0, 1 )$ be arbitrary. We now show that $\mathcal{E}_f$ is $(\alpha_f,\beta_f)$-strong winning.  Define
\begin{equation*}
	\beta_F := \exp ( -C_2 ) \beta_f \in (0,1).
\end{equation*}
We set up two $( \alpha, \beta )$ games. Alice and Bob will play the primary $( \alpha_f, \beta_f )$ game on $( [ 0,1 ], \mathcal{E}_f )$, and Alicia and Bobby will play an auxiliary $( \alpha_F, \beta_F )$ game on $( [ r_1,1 ], \mathcal{E}_F )$.

The main game begins as Bob chooses $B_1 \subset [ 0,1 ]$.  Alice chooses $A_1$ such that $0 \notin A_1$.  Bob chooses $B_2$.  Alice plays arbitrarily until Bob chooses an interval that is contained in some $[ r_{n+1}, r_n ]$.  This will eventually happen for the following reason.  There are finitely many intervals $[ r_{n+1}, r_n ]$ that intersect $B_2$ (because $0 \notin B_2$), and Alice can force $| B_n | \searrow 0$ by always choosing an interval $A_n$ of length $\alpha_f | B_n |$.  Furthermore $\alpha_f < \frac{1}{2}$ and so Alice may always choose $A_n$ so as to avoid any given point in $B_n$.  After relabeling we may therefore assume without loss of generality that $B_1 \subset [ r_{N+1}, r_N ]$ for some $N \in \mathbb{N}$.

The auxiliary game begins as Bobby chooses $B_1' = f^N ( B_1 ) \subset [ r_1,1 ]$.  Alicia, as part of her winning strategy, chooses $A_1' \subset B_1'$.  Define $A_1 = f^{-N} ( A_1' ) \cap [ r_{N+1}, r_N ] \subset B_1$.  By the mean-value theorem there exist $\xi,\xi' \in B_1^\circ$ such that, using Theorem \ref{thm:youngdist},
\begin{equation*}
\frac{| A_1 |}{| B_1 |} = \frac{| A_1' |/ ( f^N )'( \xi )}{| B_1' | / ( f^N )'( \xi' )} \geq \exp \bigg( -\frac{C_2}{r_0-r_1} | f^N \xi - f^N \xi' | \bigg) \alpha_F \geq \alpha_f.
\end{equation*}
Thus $A_1$ is a permissible interval for Alice to choose; she does so.

Suppose the four players have chosen intervals $\{ A_i, B_i, A_i', B_i'  \}_{i=1}^k$ for some $k \in \mathbb{N}$ in such a way that $f^N ( B_k ) = B_k'$ and $A_k = f^{-N} ( A_k' ) \cap [ r_{N+1}, r_N ]$, and the sets $\{ A_k' \}$ are chosen as part of Alicia's winning strategy.  Bob chooses $B_{k+1} \subset A_k$.  Define $B_{k+1}' = f^N ( B_{k+1} ) \subset A_k'$.  By the mean-value theorem there exist $\eta, \eta' \in A_k^\circ$ such that, again using Theorem \ref{thm:youngdist},
\begin{equation*}
\frac{| B_{k+1}' |}{| A_k' |} = \frac{| B_{k+1} |}{| A_k |} \frac{( f^N )'( \eta )}{( f^N )'( \eta' )} \geq \exp \bigg( -\frac{C_2}{r_0-r_1} | f^N \eta - f^N \eta' | \bigg) \beta_f \geq \beta_F.
\end{equation*}
Thus $B_{k+1}'$ is a permissible interval for Bobby to choose; he does so.  Alicia, as part of her winning strategy, chooses $A_{k+1}' \subset B_{k+1}'$.  Define $A_{k+1} = f^{-N} ( A_{k+1}' ) \cap [ r_{N+1}, r_N ] \subset B_{k+1}$.  By the mean-value theorem there exist $\upsilon,\upsilon' \in B_{k+1}^\circ$ such that, using Theorem \ref{thm:youngdist} once more,
\begin{equation*}
\frac{| A_{k+1} |}{| B_{k+1} |} = \frac{| A_{k+1}' | / ( f^N )'( \upsilon )}{| B_{k+1}' | / ( f^N )' ( \upsilon' )} \geq \exp \bigg( -\frac{C_2}{r_0-r_1} | f^N \upsilon - f^N \upsilon' | \bigg) \alpha_F \geq \alpha_f.
\end{equation*}
Thus $A_{k+1}$ is a permissible interval for Alice to choose; she does so.

Both games are now complete.  Define $\{ \omega \} = \bigcap_{k=1}^\infty B_k$ and $\{ \omega' \} = \bigcap_{k=1}^\infty B_k'$.  By construction, Alicia wins; thus there exists $L \in \mathbb{N}$ such that the orbit of $\omega'$ under $F$ stays outside the interval $[ r_1, p_L )$.  Define $M := 2 + \max \{ L,N \}$.  We claim that the orbit of $\omega$ under $f$ stays outside the interval $[ 0, r_M )$.

Suppose otherwise.  Write $\omega' \in J_{m_1 m_2 \dots}$ and let
\begin{equation*}
\tau := \min \{ t \in \mathbb{N} \cup \{ 0 \} \colon f^t \omega \in [ 0, r_M ) \}.
\end{equation*}
Because $M > N+1$ and $\omega \in [ r_{N+1},r_N ]$ we have $\tau > N$.  Find $j \geq 0$ and $0 \leq s < m_{j+1}$ such that
\begin{equation*}
\tau = N + m_1 + \dots + m_j +s.
\end{equation*}
Because the orbit of $\omega'$ under $F$ avoids $[ r_1, p_L )$ we have that $m_i \leq L < M$ for all $i$.  Therefore
\begin{equation*}
F^{j+1} \omega' = f^{m_1 + \dots m_{j+1} + N} \omega = f^{m_{j+1}-s} ( f^\tau \omega ) \in [ 0, r_{M-m_{j+1}+s} ) \subset [ 0,r_1 ).
\end{equation*}
But $F^{j+1} \omega' \in J_{m_{j+2}} \subset [ r_1,1 ]$, a contradiction. So $\omega \in \mathcal{E}_f$ and $\mathcal{E}_f$ is $(\alpha_f,\beta_f)$-strong winning. As $\beta_f$ was arbitrary, this shows that $\mathcal{E}_f$ is $\alpha_f$-strong winning.

\section{The case when $f$ is decreasing on $(r_1,1]$} \label{sec:decreasing}
As we mentioned in the introduction, we have assumed throughout this article that $f$ is increasing on $(r_1,1]$ for the sake of simplicity.  Now suppose that $f$ is decreasing on $(r_1,1]$.  The Markov structure of $F$ is now more complicated than in the increasing case.  In particular, an easy induction argument shows that the ordering of the basic intervals of $F$ of odd generation numbers is reversed.  More formally, if $m_1 \dots m_n < m_1^\ast \dots m_n^\ast$ in the lexicographic ordering and $n$ is even, we have
\begin{equation*}
\partial^\ell J_{m_1 \dots m_n} > \partial^\ell J_{m_1^\ast \dots m_n^\ast}
\end{equation*}
as before; but if $n$ is odd, we have instead
\begin{equation*}
\partial^\ell J_{m_1 \dots m_n} < \partial^\ell J_{m_1^\ast \dots m_n^\ast}.
\end{equation*}

The following changes are necessary to handle this added complexity.  Intervals $[a,b]$ for which $a>b$ should be interpreted to mean $[b,a]$.  All derivative estimates should be interpreted using absolute values.  In Definition \ref{def:pn} we have $p_n \nearrow 1$ instead of $p_n \searrow r_1$.  In Definition \ref{def:basicinterval} one now has $G_1 = \{ [r_1,p_1], [p_1,p_2], \dots \}$.  The last sentence of Definition \ref{def:Jsigma} should reflect the reversed ordering.  Definition \ref{def:leftend} needs to be expanded to include right endpoints of generation $n$.  If $J_\sigma$ is a basic interval of odd generation number, all expressions of the form $\partial^\ell J_\sigma + \zeta \lvert J_\sigma \rvert$ must be replaced with $\partial^r J_\sigma - \zeta \lvert J_\sigma \rvert$.  Expressions such as $[J_\sigma)$ become $(J_\sigma]$, with the obvious modifications to similar half-open intervals.  The properties of basic intervals listed after Definition \ref{def:commensurability} need obvious changes when the generation number is odd.  In Lemma \ref{lem:accum}, interchange the words ``left'' and ``right''.

Finally, recall that the induction step of the proof of Theorem \ref{thm:maintheoremF} appearing in \S\ref{subsec:inductionstep} is divided into two main cases, depending on whether the right half of Bob's chosen interval $B_{n_j}$ contains a left endpoint of generation at most $g_j-1$.  If $g_j$ is odd, no change is necessary, but if $g_j$ is even, the statement of Case 1 should now read: ``The interval $\big[ \partial^\ell B_{n_j}, B_{n_j}^\textrm{mid} \big)$ does not contain a right endpoint of generation at most $g_j-1$''.  Case 2 handles the complementary situation.  After these changes, the proof of the induction step is identical to the increasing case: all uses of adjacency, commensurability, and the estimates in Proposition \ref{prop:C5} are applied at the endpoint toward which the relevant descendants accumulate.  Thus no new estimate is required.

\section*{Acknowledgments}
The author is indebted to the referees for their valuable suggestions.  In addition the author would like to thank his Ph.D. advisor, Vaughn Climenhaga, for his infinite patience and wisdom.  This work is partially supported by NSF grant DMS-1554794.

\bibliography{MP}
\bibliographystyle{alpha}

\end{document}